\documentclass[]{article}
\usepackage{amsmath,amssymb,amsthm,amsfonts,cite}
\usepackage{latexsym}
\usepackage[utf8]{inputenc}
\usepackage[english]{babel}

\usepackage[T1]{fontenc}
\usepackage{cite}

\newtheorem{thm}{Theorem}
\newtheorem{lemma}{Lemma}
\newtheorem{cor}{Corollary}

\DeclareMathOperator{\co}{co}
\DeclareMathOperator{\clco}{\overline{co}}

\title{Time-Optimal Control\\of Linear Fractional Systems}
\author{Ivan Matychyn}
\date{}
\begin{document}
\maketitle

\begin{abstract}
Problem of time-optimal control of linear systems with fractional dynamics is treated in the paper from the convex-analytic standpoint. A linear system of fractional differential equations involving Riemann--Liouville derivatives is considered. A method to construct a control function that brings trajectory of the system to the terminal state in the shortest time is proposed in terms of attainability sets and their support functions.
\end{abstract}

\section{Introduction}

Optimal control of systems with fractional dynamics is a hard problem due to specific of fractional differentiation operators, e.g. lack of the semigroup property. The papers on this topic include \cite{agrawal}, where necessary optimality conditions of Euler--Lagrange were derived, and \cite{TricChen}, where the problem of time-optimal control is addressed.

Here the fractional time-optimal control problem \cite{TricChen} for a linear system with fractional dynamics is treated using technique of attainability sets and their support functions. This approach has its roots in some methods of the differential games theory \cite{Chik,PshOst}.

\section{Preliminary Results}

Denote by $ \mathbb{R}^n $ the $ n$-dimensional Euclidean space and by $\mathbb{R}_+=[0,\infty)$ the positive semi-axis. In what follows we will also denote by $ x\cdot y $ the scalar (dot) product and by $\|x\|$ the Euclidean norm for any $ x,y\in\mathbb{R}^n $.

Suppose $ f:\mathbb{R}_+\to\mathbb{R}^n$ is an absolutely continuous function. Let us recall that the Riemann--Liouville (left-sided) fractional integral and derivative of order $\alpha$, $0<\alpha<1$, are defined as follows:
\begin{align*}
J_{a+}^\alpha f(t) &= \frac{1}{\Gamma(\alpha)}\int_{a}^{t} (t-\tau)^{\alpha-1}f(\tau)d\tau,\ t> a,  \\
D_{a+}^\alpha f(t) &= \frac{d}{dt} J_{a+}^{1-\alpha} f(t),\ t>a.
\end{align*}
In what follows we will omit the lower limit of integration in the notation if it is equal to zero, i.e. $ J^\alpha f(t) \triangleq J_{0+}^\alpha f(t) $, $ D^\alpha f(t) \triangleq D_{0+}^\alpha f(t)$.

Along with the left-sided fractional integrals and derivatives, one can consider their right-sided counterparts:
\begin{align*}
J_{b-}^\alpha f(t) &= \frac{1}{\Gamma(\alpha)}\int_{t}^{b} (\tau-t)^{\alpha-1}f(\tau)d\tau,\ t<b, \\
D_{b-}^\alpha f(t) &= -\frac{d}{dt} J_{b-}^{1-\alpha} f(t),\ t<b.
\end{align*}

In \cite{ChikEid} the Mittag-Leffler generalized matrix function was introduced:
\begin{equation} \label{ML}
E_{\rho,\mu}(B)=\sum\limits_{k=0}^\infty\frac{B^k}{\Gamma(k\rho+\mu)},
\end{equation}
where $\rho > 0$, $\mu\in\mathbb{C}$, and $B$ is an arbitrary square matrix of order $n$. It should be noted that $E_{\rho,\mu}(B)$ generalizes the matrix exponential as
\begin{equation} \label{expo}
E_{1,1}(B)=e^B=\sum\limits_{k=0}^\infty\frac{B^k}{k!}.
\end{equation}  
The matrix $\alpha$-Exponential function, introduced in \cite{KilSriTru}, is closely related to the Mittag-Leffler generalized matrix function:
\begin{equation*}
e_\alpha^{At}=t^{\alpha-1}\sum\limits_{k=0}^{\infty} \frac{A^k t^{\alpha k}}{\Gamma[(k+1)\alpha]}=t^{\alpha-1} E_{\alpha,\alpha}(At^\alpha).
\end{equation*}
The both functions play important role in the theory of fractional differential equations (FDEs). In particular, consider a system of linear FDEs with constant coefficients
\begin{equation} \label{fde}
D^\alpha z = Az + u,\quad 0<\alpha<1,
\end{equation}
where $z\in\mathbb{R}^n$, $ A $ is a square matrix, and $ u:\mathbb{R}_+\to \mathbb{R}^n$ is a measurable and bounded function, under the initial condition
\begin{equation} \label{fdeini}
J^{1-\alpha}z\bigr |_{t=0}=z^0.
\end{equation}
Then the solution to the Cauchy-type problem \eqref{fde}, \eqref{fdeini} can be written down as follows \cite{ChikMat1}
\begin{equation} \label{fdesoln1}
z(t)=t^{\alpha-1} E_{\alpha,\alpha}(At^\alpha)z^0 + \int_{0}^{t}(t-\tau)^{\alpha-1}E_{\alpha,\alpha} (A(t-\tau)^\alpha)u(\tau)d\tau
\end{equation}
or, in terms of matrix $ \alpha $-exponential function, as \cite{KilSriTru}
\begin{equation} \label{fdesoln2}
z(t) = e_\alpha^{At} z^0 + \int_{0}^{t} e_\alpha^{A(t-\tau)}u(\tau)d\tau.
\end{equation}

Now we proceed with a homogeneous linear system involving right-sided fractional derivative in the sense of Riemann--Liouville
\begin{equation} \label{sopr}
D_{b-}^\alpha z(t)=Az(t),\quad z\in\mathbb{R}^n,\ t<b,\ 0<\alpha<1,
\end{equation}
under the boundary condition
\begin{equation} \label{bnd}
J_{b-}^{1-\alpha} z \bigr |_{t=b}=\hat{z}.
\end{equation}

The following lemma holds true.

\begin{lemma} \label{adjlm}
Equation \eqref{sopr} under the condition \eqref{bnd} has a solution given by the following formula
\begin{equation} \label{soprsln}
z(t)=\hat{z}e_\alpha^{A(b-t)}.
\end{equation}
\end{lemma}
\begin{proof}
Since $ e_\alpha^{A(b-t)} $ is an entire function, the corresponding power series can be integrated and differentiated term-by-term. In view of the formulas \cite{SamKilMar}
\begin{align}
D_{b-}^\alpha (b-t)^{\beta - 1} &= \frac{\Gamma(\beta)}{\Gamma(\beta - \alpha)}(b-t)^{\beta-\alpha-1},\label{diff}\\
J_{b-}^\alpha (b-t)^{\beta - 1} &= \frac{\Gamma(\beta)}{\Gamma(\beta + \alpha)}(b-t)^{\beta+\alpha-1},\label{integ}
\end{align}
one can easily verify that \eqref{soprsln} satisfies \eqref{sopr} by direct substitution. Moreover, \eqref{integ} implies that $ J_{b-}^{1-\alpha} \hat{z}e_\alpha^{A(b-t)} = \hat{z}E_{\alpha,1}(A(b-t)^\alpha)$ and condition \eqref{bnd} is also fulfilled.
\end{proof}

The following properties of the matrix $ \alpha $-exponential function are direct consequences of  properties  of the conjugate transpose:
\begin{equation} \label{pr1}
(e_\alpha^{At})^\ast=e_\alpha^{A^\ast t}
\end{equation}
\begin{equation} \label{pr2}
\psi \cdot e_\alpha^{At}u = e_\alpha^{A^\ast t}\psi \cdot u
\end{equation}

Below we present some properties \cite{ArkLev} of set-valued maps  used in the sequel.

Let us suppose that $ U $ is a nonempty compact (closed and bounded) set in $ \mathbb{R}^n $. Hereafter we will denote by $U[0,t]$ the set of all measurable functions defined on $ [0,t] $ and taking their values in $ U $.
\begin{equation} \label{pr3}
\sup\limits_{u(\cdot)\in U[0,t]}\int_0^t f(\tau,u(\tau)) d\tau= \int_{0}^{t}\max\limits_{u\in U}f(\tau,u)d\tau
\end{equation}

Denote by $ \co M $ and $ \clco M$  the convex hull and the closure of the convex hull of a set $ M \subset\mathbb{R}^n$, respectively.

For any continuous function $ F:[0,t] \times U \to \mathbb{R}^n $, the set-valued map $ F(\tau,U) $ possess the following property:
\begin{equation} \label{pr4}
\int_{0}^{t} F(\tau,U)d\tau=\left\{\int_{0}^{t} F(\tau,U)d\tau:\ u(\cdot)\in U[0,t]   \right\}=\clco \int_{0}^{t} F(\tau,U)d\tau.
\end{equation}
The integral $ \int_{0}^{t} F(\tau,U)d\tau$ is to be thought of in the sense of Aumann, i.e. as the set of integrals of all measurable selections of the set valued map $F(\tau,U)$.

Here we recall definition of the support function and present a useful result of convex analysis.

Let $ M\in\mathbb{R}^n $ be a convex closed set, i.e. $ M=\clco M$. Then the function 
\begin{equation*}
\sigma_M(\psi)=\sup\limits_{m\in M} \psi\cdot m, \quad \psi\in\mathbb{R}^n
\end{equation*}
is called the support function of $ M $.

\begin{lemma}[\cite{PshOst}]
Let $ X $ and $ M $ be convex closed sets. Moreover, assume that $ X $ is bounded. Then $ X\cap M=\emptyset $ if and only if there exist a vector $ \psi \in \mathbb{R}^n $ and a number $ \varepsilon>0 $ such that
\begin{equation}
\sigma_X(\psi) + \sigma_M(-\psi)\le -\varepsilon.
\end{equation}
\end{lemma}
\begin{proof}
Since the sets $ X $ and $ M $ are disjoint, by assumptions of the lemma and in view of hyperplane separation theorem, there exist a vector $ \psi $ and a number $ \varepsilon>0$ such that 
\begin{equation*}
\psi\cdot x \le \psi\cdot m-\varepsilon\quad \forall x\in X,\ m\in M.
\end{equation*}
Taking supremum in $ x $  on the left-hand side and infimum in $ m $ on the right-hand side, we obtain an equivalent inequality
\begin{equation*}
\sigma_X(\psi)\le \inf_{m\in M} \psi\cdot m - \varepsilon = -\sup _{m\in M} (-\psi\cdot m) - \varepsilon = -\sigma_M(-\psi) - \varepsilon,
\end{equation*}
which completes the proof.
\end{proof}

\begin{cor} \label{cr1}
Let $X=\clco X$, $ M=\clco M$ and $ X $ be bounded. Then $ X\cap M\ne\emptyset $ if and only if 
\begin{equation}
\lambda_{X,M}=\min\limits_{\|\psi\|=1}[\sigma_X(\psi)+\sigma_M(-\psi)]\ge 0.
\end{equation}
\end{cor}

\section{Optimal Control Problem}
Consider a system of linear fractional differential equations (FDEs) with constant coefficients \eqref{fde} under the initial condition \eqref{fdeini}.

Let us fix a point $ m\in \mathbb{R}^n $. Here we formulate the optimal control problem: find a control function $ u(\cdot) $, $ u:\mathbb{R}_+\to U $, from a class of measurable functions taking their values in a nonempty compact set $ U $, $ U\subset \mathbb{R}^n $, such that the corresponding trajectory of \eqref{fde}, \eqref{fdeini} arrives at $ m $ in the shortest time $ T $. 

If we fix some admissible control function $ u(\cdot)\in U[0,t]$, then the solution to the Cauchy-type problem \eqref{fde}, \eqref{fdeini} is given by  \eqref{fdesoln2}.

Consider the attainability set 
\begin{equation} \label{fdeattn}
\begin{aligned}
Z(t,z^0) &= \left\{e_\alpha^{At} z^0 + \int_{0}^{t} e_\alpha^{A(t-\tau)}u(\tau)d\tau:\ u(\cdot)\in U[0,t]  \right\}\\
&= e_\alpha^{At} z^0 + \int_{0}^{t} e_\alpha^{A(t-\tau)} U d\tau.
\end{aligned}
\end{equation}
According to properties of integrals of set-valued maps, in view of \eqref{pr4}, the attainability set $ Z(t,z^0) $ is closed and convex, while the boundedness of $ U $ implies $ Z(t,z^0) $ is also bounded.

Consider support function of the attainability set \eqref{fdeattn}.
\begin{equation} \label{fdesup}
\begin{gathered}
\sigma_{Z(t,z^0)}(\psi)=\sup\limits_{z\in Z(t,z^0)}(z\cdot \psi)\\
=\sup\limits_{u(\cdot)\in U[0,t]}\left\{ \psi \cdot e_{\alpha}^{At}z^0 + \int_{0}^{t} \psi\cdot e_{\alpha}^{A(t-\tau)} u(\tau) d\tau\right\}\\
= \psi \cdot e_{\alpha}^{At}z^0 + \int_{0}^{t} \sigma_U(e_{\alpha}^{A^\ast(t-\tau)}\psi)  d\tau.
\end{gathered}
\end{equation}
Here we applied properties \eqref{pr1}--\eqref{pr3}.

Let us introduce the function 
\begin{equation} \label{lambda}
\lambda(t,z^0)=\min\limits_{\|\psi\|=1}[\sigma_{Z(t,z^0)}(\psi)-m\cdot \psi]
\end{equation}
and denote
\begin{equation} \label{time}
T(z^0)=\min \{ t\ge 0: \ \lambda(t,z^0)\ge 0\}.
\end{equation}

Then the following theorem holds true.
\begin{thm}
Trajectory of the system \eqref{fde}, \eqref{fdeini} can be brought to the point $ m $ at the minimal time $ T=T(z^0) $, given by the formula \eqref{time}, with the help of control function of the form
\begin{equation*}
\hat{u}(\tau)=\arg\max\limits_{u\in U} u\cdot \psi(\tau),
\end{equation*}
where $ \psi(\tau)$ is a solution to the adjoint (co-state) system
\begin{align} 
D_{T-}^\alpha\psi &= A^\ast \psi, \label{adj1}\\
J_{T-}^{1-\alpha}\psi \bigr |_{t=T} &= \hat{\psi} \label{adj2}
\end{align}
and $ \hat{\psi}=\arg\min\limits_{\|\psi\|=1} [\sigma_{Z(T,z^0)}(\psi)-\psi\cdot m]$.
\end{thm}
\begin{proof}
Let $ T=\min\{t\ge 0:\ m\in Z(t,z^0)\} $. Here minimum is attained due to the closedness of $ Z(t,z^0) $. 

Moreover, $ m $ is a boundary point of $ Z(T,z^0) $, i.e. $m\in\partial Z(T,z^0) $. As a boundary point, $m$ is contained in at least a supporting hyperplane $ H(\hat{\psi})=\{x\in\mathbb{R}^n:\ \hat{\psi}\cdot x = \sigma_{Z(T,z^0)}(\hat{\psi})\} $. Hence, for some $ \hat{\psi} $
\begin{equation} \label{eq1}
\hat{\psi}\cdot m = \sigma_{Z(T,z^0)}(\hat{\psi}).
\end{equation}
Thus, the control function $\hat{u}(\cdot)$ that ensures bringing trajectory of \eqref{fde}, \eqref{fdeini} to the point $ m $ is the function at which the maximum in \eqref{fdesup} is attained. Therefore it must satisfy 
\begin{equation*}
\hat{u}(\tau)=\arg\max\limits_{u\in U} u\cdot e_{\alpha}^{A^\ast(T-\tau)}\hat{\psi},\ \tau \in [0,T].
\end{equation*}
In view of Lemma \ref{adjlm}, $ \psi(\tau) = e_{\alpha}^{A^\ast(T-\tau)}\hat{\psi}$ is a solution to \eqref{adj1}.

According to Corollary \ref{cr1}, $m\in Z(t,z^0)$ if and only if $\lambda(t,z^0)\ge 0 $, hence $ T=T(z^0)=\min \{ t\ge 0: \ \lambda(t,z^0)\ge 0\} $.
Since $ \lambda(T,z^0)\ge 0 $, in virtue of \eqref{eq1}, $\hat{\psi}$ delivers minimum to the expression $ \sigma_{Z(T,z^0)}(\psi)-\psi\cdot m $.

Thus \begin{equation*}
\hat{u}(\tau)=\arg\max\limits_{u\in U} u\cdot \psi(\tau),
\end{equation*}
where $ \psi(\tau)$ is a solution to \eqref{adj1} under the condition \eqref{adj2}, which completes the proof.
\end{proof}

\section{Example}

Let us illustrate the above results by a simple example.

Consider a system with fractional dynamics described by the equation
\begin{equation} \label{ex1}
D^\alpha z=u,\ z\in \mathbb{R}^n,\ \|u\|\le 1,\ 0<\alpha<1,
\end{equation}
under the initial condition \eqref{fdeini}. In this example, the matrix $A$ and $U$ is the unit ball centered at the origin.  

Hence
\begin{equation*}
e_\alpha^{At}=e_\alpha^{A^\ast t}=\frac{t^{\alpha-1}}{\Gamma(\alpha)}I
\end{equation*}
and the support function of the attainability set has the form
\begin{gather*}
\sigma_{Z(t,z^0)}(\psi) = \psi \cdot e_{\alpha}^{At}z^0 + \int_{0}^{t} \sigma_U(e_{\alpha}^{A^\ast(t-\tau)}\psi)  d\tau\\
= \psi \cdot \frac{t^{\alpha-1}}{\Gamma(\alpha)}z^0 + \int_{0}^{t} \left\|\frac{(t-\tau)^{\alpha-1}}{\Gamma(\alpha)} \psi \right\|  d\tau\\
= \psi \cdot \frac{t^{\alpha-1}}{\Gamma(\alpha)}z^0 + \|\psi\| \frac{t^{\alpha}}{\Gamma(\alpha+1)}.
\end{gather*}
Suppose that $m=0$, then
\begin{equation} \label{ex2}
\lambda(t,z^0)=\min\limits_{\|\psi\|=1}[\sigma_{Z(t,z^0)}(\psi)]=\frac{t^{\alpha}}{\Gamma(\alpha+1)}-\frac{t^{\alpha-1}}{\Gamma(\alpha)}\|z^0\|.
\end{equation}
Thus the minimum time $T$, at which trajectory of the system \eqref{ex1} can reach the origin can be found as the smallest positive root of the equation
\[
\frac{t^{\alpha}}{\Gamma(\alpha+1)}=\frac{t^{\alpha-1}}{\Gamma(\alpha)}\|z^0\|.
\]
At $ t=0 $ the left-hand side of the latter equation is zero while its right-hand side is infinite, provided that $ \|z^0\|\ne 0 $. As $ t\to\infty $ the left-hand side increases without bound and the right-hand side approaches zero. Thus, the equation has a positive solution.

The system adjoint to \eqref{ex1}, \eqref{fdeini} has the form
\begin{align*} 
D_{T-}^\alpha\psi &= 0, \\
J_{T-}^{1-\alpha}\psi \bigr |_{t=T} &= -\frac{z^0}{\|z^0\|}
\end{align*}
and its solution is
\[
\psi(t)=-\frac{(T-t)^{\alpha-1}}{\Gamma(\alpha)}\frac{z^0}{\|z^0\|}.
\]
Finally the optimal control that ensures bringing trajectory of \eqref{ex1}, \eqref{fdeini} to the origin at the minimal time $ T$ is
\[
u(t) \equiv -\frac{z^0}{\|z^0\|}.
\]

\bibliographystyle{unsrt}

\end{document}